\documentclass[11pt]{article}
\usepackage{epsfig, amsmath, amssymb, amsthm, times,color}

\usepackage[utf8]{inputenc}

\usepackage{amsmath,amssymb,amsmath}
\usepackage{mathabx}
\usepackage{graphicx}
\usepackage{subfigure}
\usepackage{cite}
\usepackage{hyperref}
\usepackage{url}
\usepackage{mathrsfs}  

\usepackage[T1]{fontenc}
\usepackage{esint}

\numberwithin{equation}{section}

\newcommand{\R}{\mathbb{R}}
\newcommand{\T}{\mathbb{T}}

\newcommand{\Z}{\mathbb{Z}}
\newcommand{\N}{\mathbb{N}}

\newcommand{\cH}{\mathcal{H}}

\newcommand{\Lip}{\operatorname{Lip}}

\def\lbl{\label}
\def\be{\begin{equation}}
\def\ee{\end{equation}}

\newcommand{\p}{{\partial}}

\def\1{\mathbf{1}}

\newtheorem{theorem}{Theorem}[section]

\newtheorem{lemma}[theorem]{Lemma}

\newtheorem{remark}[theorem]{Remark}
\newtheorem{definition}[theorem]{Definition}

\newtheorem{example}[theorem]{Example}

\DeclareMathOperator{\esssup}{ess\,sup}

\title{Galerkin method for nonlocal diffusion equations on self-similar domains
}
\date{\today}
\author{
Georgi S. Medvedev\thanks{Department of Mathematics, 
Drexel University, 3141 Chestnut Street, Philadelphia, PA 19104,
{\tt medvedev@drexel.edu}}
}

\begin{document}
\maketitle
\begin{abstract}
  Integro--differential equations, analyzed in this work, comprise an important
  class of models of continuum media with nonlocal interactions. Examples include peridynamics,
  population and opinion dynamics, the spread of disease models, and nonlocal diffusion, to name a few.
  They also arise naturally as a continuum limit of interacting dynamical systems on networks.
  Many real-world networks, including  neuronal, epidemiological, and information networks,
  exhibit self-similarity, which translates into self-similarity of the spatial domain of
  the continuum limit.

  For a class of evolution equations with nonlocal interactions on self-similar domains, we construct
  a discontinuous Galerkin method and develop a framework for studying its convergence.
  Specifically, for the model at hand, we identify a natural scale of function spaces,
  which respects self-similarity of the spatial domain, and estimate the rate of convergence
  under minimal assumptions on the regularity of the interaction kernel. The analytical
  results are illustrated by numerical experiments on a model problem.
\end{abstract}

\section{Introduction}
\setcounter{equation}{0}

For a class of interacting dynamical systems on a convergent graph sequence, one can write
a continuum limit in the form of a nonlocal diffusion equation \cite{Med14a, Med14b, Med19}
\be\lbl{nloc}
\p_t u(t,x) = f(t,u) + \int_{[0,1]} W(x,y) D\left(u(t,y), u(t,x)\right) dy, \quad x\in [0,1],
\ee
where $W\in L^2([0,1]^2)$ and $f, D$ are sufficiently regular functions. As an example, consider
the Kuramoto model of coupled phase oscillators with nonlocal nearest-neighbor coupling
\be\lbl{KM}
\dot u_i = \omega +\frac{1}{n} \sum_{j=i-k}^{i+k} \sin\left( 2\pi\left(u_j-u_i\right)\right),
\quad i \in \Z_n\doteq \Z/n\Z,
\ee
where  $k=\lfloor rn\rfloor$ with $r\in (0, 1/2)$ and $u_i:\R\to\T\doteq \R/\Z$ is the phase of oscillator $i$. 
In \cite{WilStr06}, in the limit as $n\to \infty$ \eqref{KM} was approximated by
\be\lbl{cKM}
\p_t u(t,x) = \omega +\int_\T K(x-y) \sin\left( 2\pi\left( u(t,x)-u(t,y)\right)\right) dy,\quad x\in\T,
\ee
where $K(x)$ is a $1-$periodic function equal to $\1_{|x|\le r}$ for $|x|\le 1/2$.

\begin{figure}
	\centering
 \textbf{a}\;	\includegraphics[width =.45\textwidth]{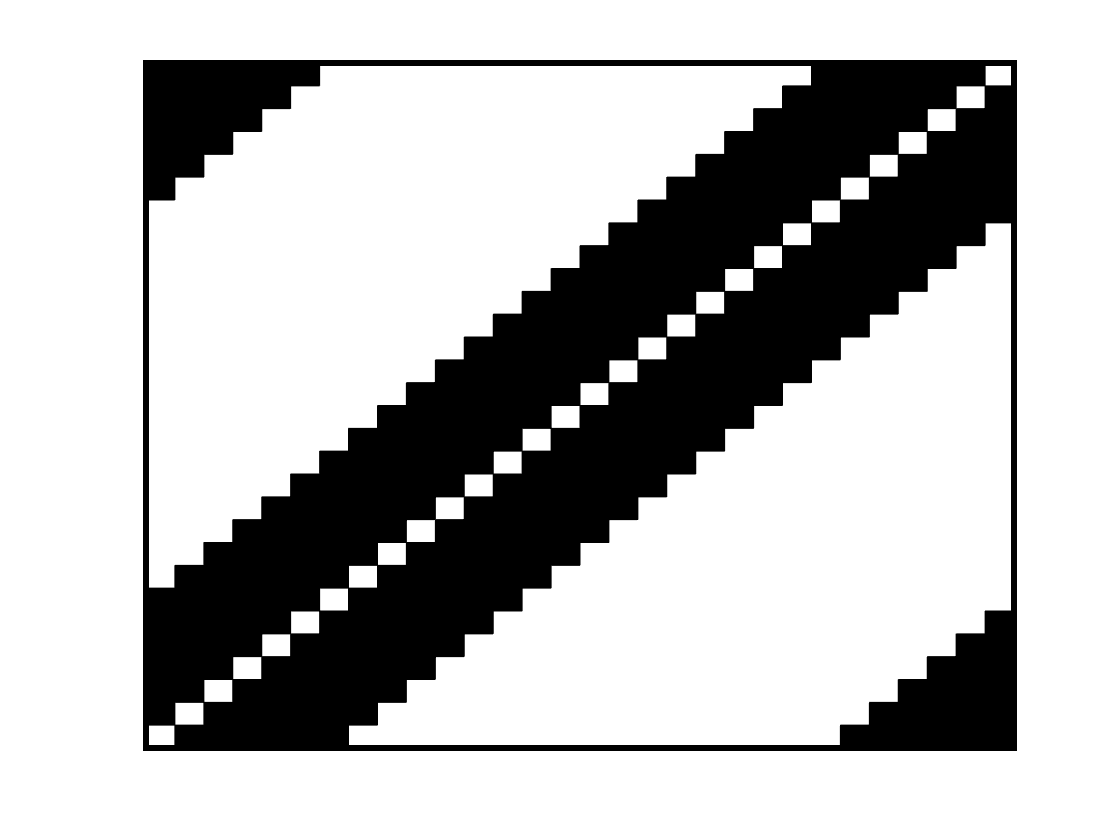}\qquad
  \textbf{b}\;      \includegraphics[width = .45\textwidth]{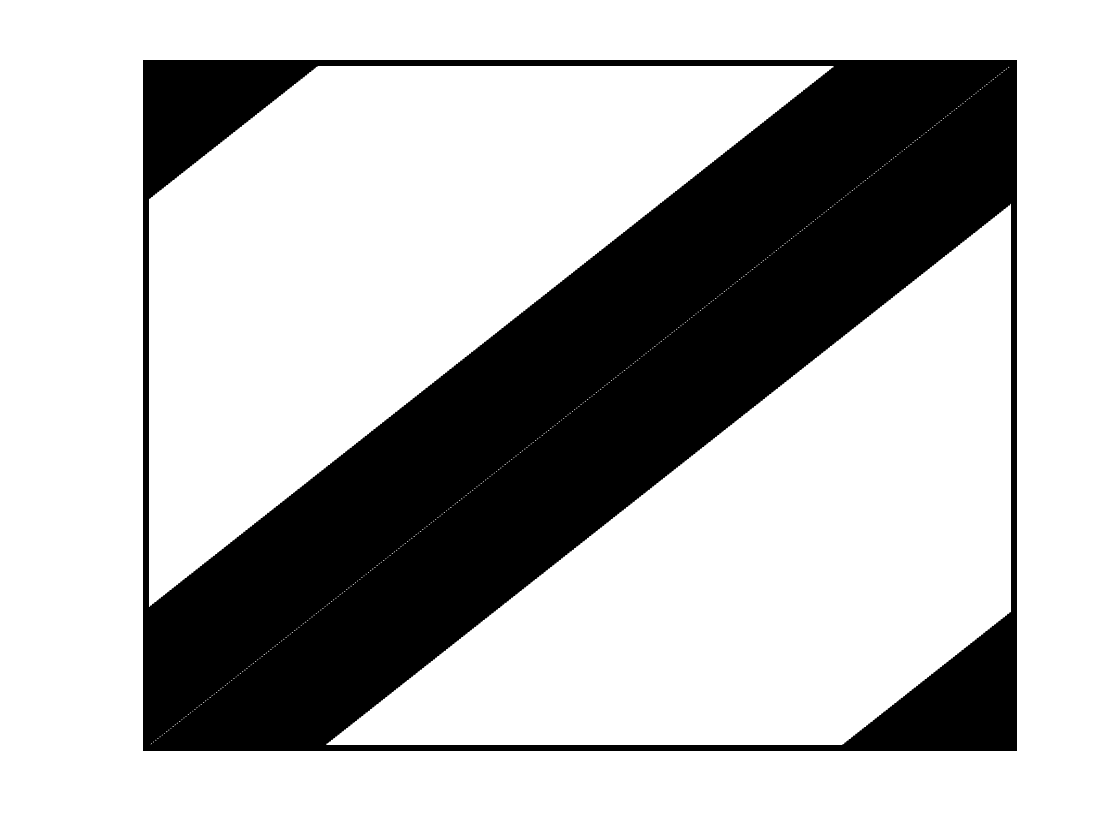}
	\caption{Support of $W^n$ (see \eqref{adjWn}) and that of $W$, the limit of $W^n$ as $n\to\infty$.}
	\label{f.1}
\end{figure}

The geometric basis of such approximation is clear from the structure
of the adjacency matrix of the $k$-nearest-neighbor graph
$$
a_{i,j}=\1_{\{d_n(i,j)\le \lfloor rn\rfloor\}}, \; d_n(i,j)=\min\{|i-j|, n-|i-j|\}.
$$
Representing this matrix by a step function on the unit square
\begin{equation}\label{adjWn}
W^n(x,y)= \sum_{i,j=1}^n a^n_{ij} \1_{\left[\frac{i-1}{n}, \frac{i}{n}\right)\times \left[\frac{j-1}{n},\frac{j}{n}\right)}(x,y),
\end{equation}
it is easy to see that as $n\to\infty$, the support of $W^n$ tends to the
support of the indicator function of $W$, the kernel in
the continuum
limit \eqref{nloc} (see Figure~\ref{f.1}). Thus, formally one expects that the sum on the right--hand side of \eqref{KM}
is transformed into the integral
on the right--hand side of the continuum limit \eqref{cKM}. A more careful analysis reveals that the relation
between
the discrete model and the continuum limit \eqref{nloc} can be interpreted with the help of the Galerkin
scheme
\cite{Med14a}.  Specifically, given the continuum limit \eqref{nloc} we discretize it by projecting onto
a finite--dimensional
subspace of the phase space $\cH\doteq L^2([0,1])$,
$\cH^n\doteq\operatorname{span}\{ \1_{\left[\frac{i-1}{n}, \frac{i}{n}\right)},\; i\in [n]\}$.
To this end, approximate $u$ by
$$
u^n(t,x) =\sum_{i=1}^n u^n_i(t) \1_{\left[\frac{i-1}{n}, \frac{i}{n}\right)}(x),
$$
plug this in \eqref{cKM} and take the inner product with $\1_{\left[\frac{i-1}{n}, \frac{i}{n}\right)}$ to obtain
\be\lbl{disc}
\dot u^n_i= f(t, u^n_i) + \frac{1}{n} \sum_{j=1}^n W^n_{ij} D(u^n_j, u_i^n), \qquad i\in [n],
\ee
where
\be\lbl{Wij}
W^n_{ij}=n^2\int_{ \left[\frac{i-1}{n}, \frac{i}{n}\right)\times\left[\frac{j-1}{n},
\frac{j}{n}\right) } W dx =\fint_{\left[\frac{i-1}{n}, \frac{i}{n}\right)} W dx.
\ee
Here and below, for $f\in L^1(Q,\mu)$, $\fint_Q f(x)d\mu(x)$ stands for $\mu(Q)^{-1} \int_Q f(x)d\mu(x)$.

It is instructive to rewrite \eqref{disc} as a PDE on $[0,1]$:
\be\lbl{d-nloc}
\p_t u^n(t,x) = f(t,u^n) + \int_{[0,1]} W^n(x,y) D\left(u^n(t,y), u^n(t,x)\right) dy, \quad x\in [0,1],
\ee
where
\be\lbl{Wn}
W^n(x,y)=\sum_{i,j=1}^n W^n_{ij} \1_{\left[\frac{i-1}{n}, \frac{i}{n}\right)\times\left[\frac{j-1}{n}, \frac{j}{n}\right) } (x,y).
\ee
Thus, the justification of the continuum limit can be rephrased as a convergence problem for
the Galerkin scheme for \eqref{nloc}. Standard estimates yield the following bound for the error of the Galerkin
scheme for the problem at hand (cf.~\cite{Med14a})
\be\lbl{error-nloc}
\sup_{t\in [0,T]} \| u(t,\cdot)-u^n(t,\cdot)\|_{L^2([0,1])}\le C\left( \|W-W^n\|_{L^2([0,1]^2)} +
  \| u(0,\cdot)-u^n(0,\cdot)\|_{L^2([0,1])}\right).
\ee
Thus, the accuracy of the continuum limit is determined by the accuracy of approximation of the graphon $W$ by its
$L^2$--projection onto $\cH^n$. For $W$, as a limit of a graph
sequence, the only natural regularity assumption is measurability
or integrability in case of the $L^p$--graphons \cite{BCCZ19}. For $W\in L^p([0,1]^2),$ $p\ge 1$,
by construction, $W^n\to W$
a.e. and in $L^p$ (see, e.g., \cite[Proposition~2.6]{Cha17}).
 However, in the absence of additional assumptions, the  convergence may be arbitrarily slow as the following
example from \cite{Med14a}
shows.
\begin{figure}
	\centering
 % \textbf{a}\;	\includegraphics[width =.45\textwidth]{SG-level-7.pdf}\qquad
 % \textbf{b}\;      \includegraphics[width = .45\textwidth]{SG-level-2.pdf}
% \textbf{a}\;	\includegraphics[width =.45\textwidth]{f2a.pdf}\qquad
% \textbf{b}\;      \includegraphics[width = .45\textwidth]{f2b.pdf}
% \textbf{a}\;	\includegraphics[width=2.5cm, height=2.0cm]{F/SG-fractal.pdf}
% \textbf{b}\;      \includegraphics[width=11cm, height=2.0cm]{F/SG-graphs.pdf}
\textbf{a}\;\includegraphics[width=5.0cm]{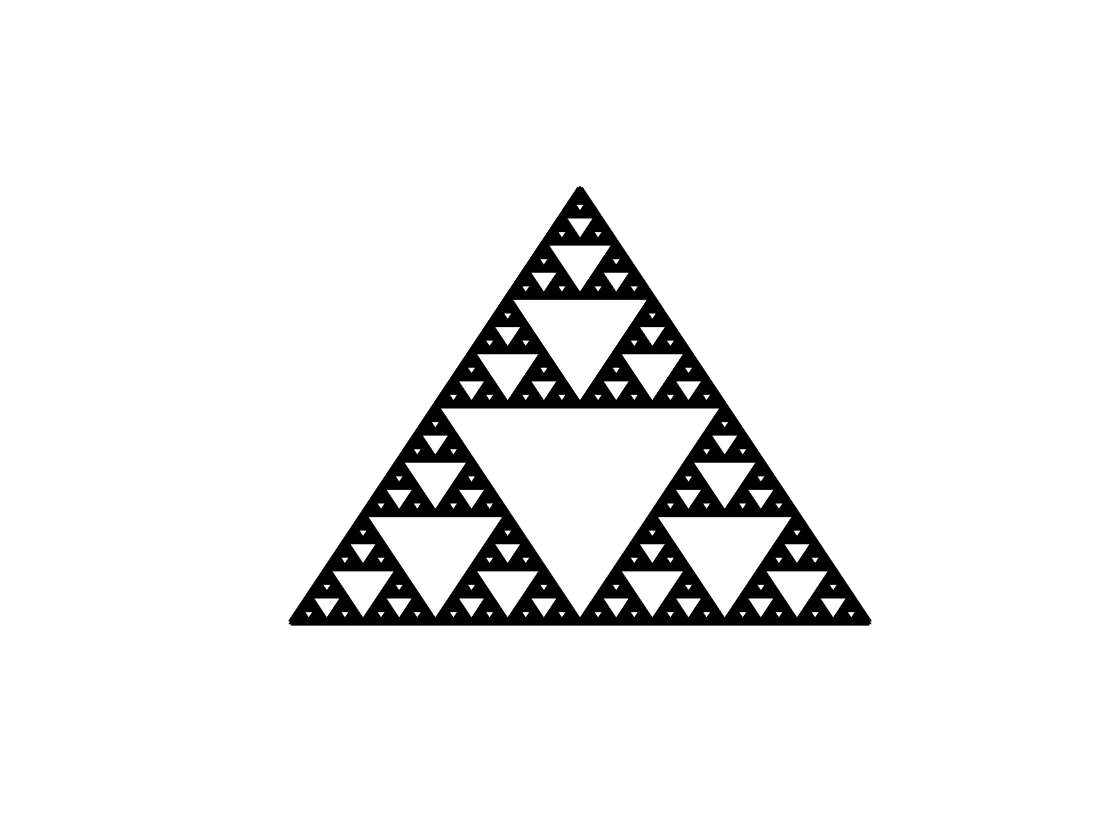}    
\textbf{b}\;\includegraphics[width=5.0cm]{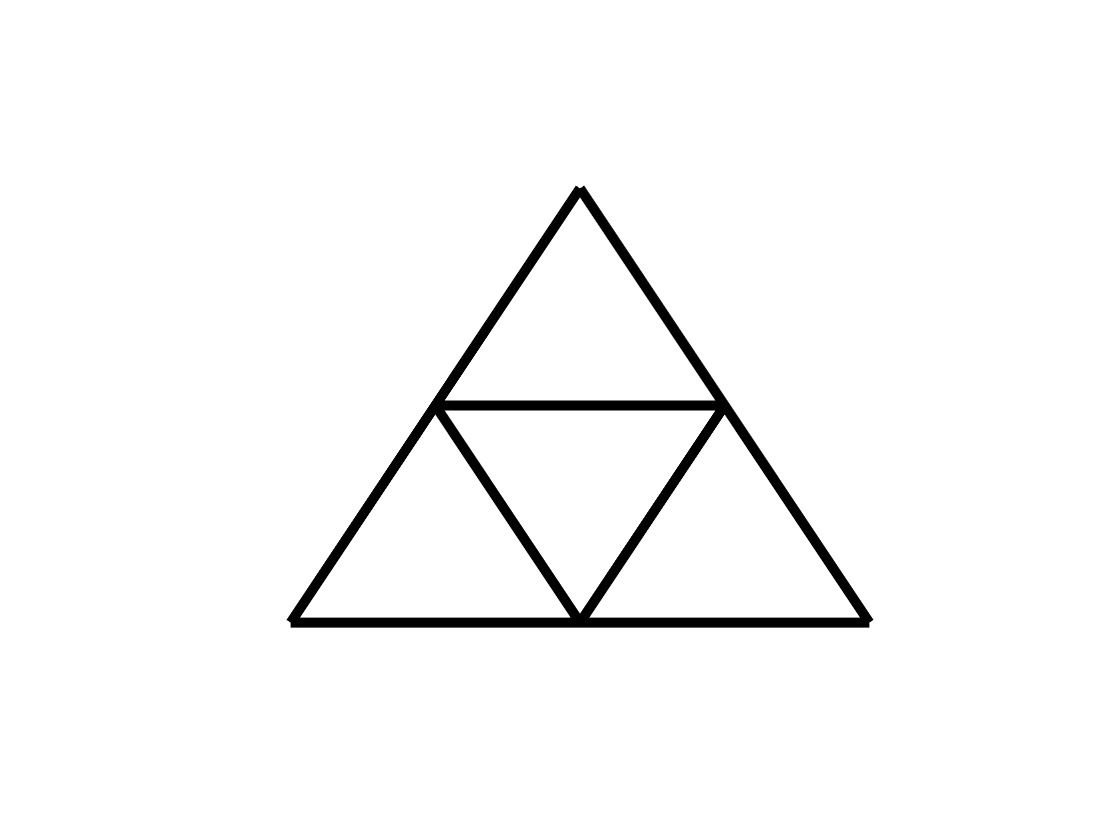}
\textbf{c}\;\includegraphics[width=5.0cm]{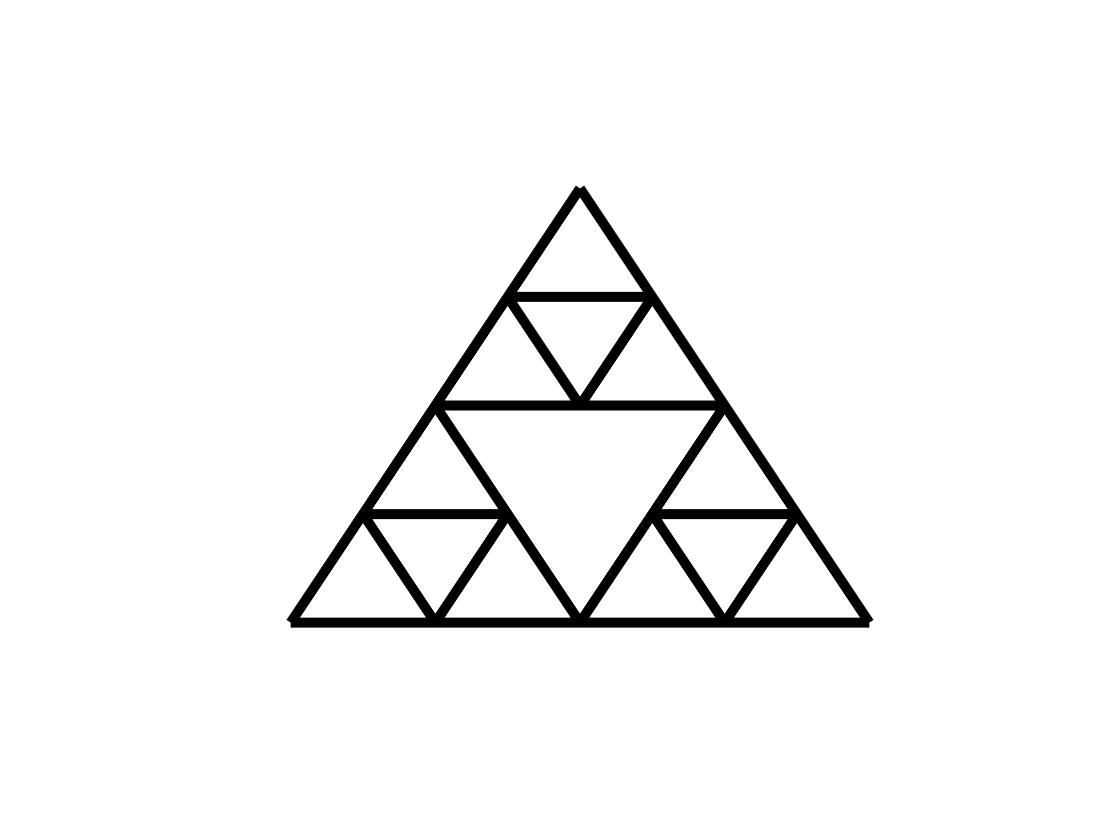}    
	\caption{\textbf{a})~ST.\;\; \textbf{b,c})~Graphs approximating ST.}
	\label{f.2}
\end{figure}

Consider $W: [0,1]^2\to \{0,1\}$ and denote by $W^+$ and $\p W^+$, the support of $W$ and its boundary respectively.
It is not hard to see that in this case
\be\lbl{error-01}
\|W-W^n\|^p_{L^p([0,1]^2)}\le  N(\p W^+, n)n^{-2},
\ee
where $N(\p W^+,n)$ is the number of discrete cells 
$\left[\frac{i-1}{n}, \frac{i}{n}\right)\times\left[\frac{j-1}{n}, \frac{j}{n}\right), i,j\in [n],$ that intersect $\p W^+$.
Let $\gamma$ stand for the upper box counting dimension $N(\p W^+,n)$ (cf.~\cite{Falc-FracGeom}),
$$
\gamma=\varlimsup_{n\to\infty} \frac{\log N(\p W^+,n)}{\log n}.
$$
Then for any $\epsilon>0$, we have
\be\lbl{bound-N}
N(\p W^+,n)\le n^{\gamma +\epsilon}.
\ee
By plugging \eqref{bound-N} into \eqref{error-01}, we have
\be\lbl{final-01}
\|W-W^n\|_{L^p([0,1]^2)}\le n^{\frac{-2+\gamma+\epsilon}{p} }.
\ee
The bound on the error of approximation suggests that the convergence can be in principle arbitrarily slow 
if the fractal dimension of the boundary of support of $W$ is sufficiently close to $2$.

This leads to the following question: What are the natural  assumptions on $W$ just beyond integrability
that allow to deduce the rate of convergence of $W^n$ to $W$?  In \cite{KVMed22}, the  generalized Lipschitz
spaces (cf.~\cite{Nikol-approximation}) were suggested for this role:
\begin{equation}\label{Lip-space}
\Lip\left( L^p(Q),\alpha\right) = \left\{
    f\in L^p(Q):\quad \omega_p (f,\delta)\le C \delta^\alpha \right\},\quad
  0<\alpha\le 1.
  \end{equation}
  Here, the $L^p$--modulus of continuity $\omega_p(f,\delta)$ is  defined as follows
  \begin{equation}\label{modulus}
  \omega_p(f,\delta) =\sup_{|h|\le \delta} \| f(\cdot)-f(\cdot+h) \|_{ L^p(Q_h^\prime)},\qquad
    Q_h^\prime =\{x \in Q: \quad x+h\in Q \},
  \end{equation}
and $|x|$ stands for the $\ell^\infty$--norm of $x\in\R^d$, and $Q$ is a given domain.

     $\Lip\left( L^p(Q),\alpha\right)$ is equipped with the norm
     \begin{equation}\label{Lip-norm}
       \|f\|_{\Lip\left( L^p(Q),\alpha\right)}=\sup_{\delta>0} \delta^{-\alpha}\omega_p(f,\delta).
     \end{equation}
 %    For the remainder of this section, for simplicity we restrict to $Q=[0,1]^d$.

    For $f\in\Lip(L^p([0,1]^d),\alpha)$, $p\ge 1$, it was shown \cite{KVMed22, MedSim22}
    \be\lbl{rate-Lip}
    \|f-f^n\|_{L^p([0,1]^d)} =O( n^{-\alpha}),
    \ee
    where $f^n$ is the $L^2$-projection of $f$ onto a subspace of piecewise constant functions
  corresponding to the rectangular discretization of $[0,1]^d$ with step $n^{-1}$.
  The proof of the estimate \eqref{rate-Lip} in \cite{KVMed22}, which is included
in the appendix to this paper for completeness,
  uses a dyadic discretization of $[0,1]^d$.
  It turns out that it extends naturally to a large class of self--similar sets.
  Specifically, in this paper we prove a counterpart of \eqref{rate-Lip} for functions defined
   on attractors of Iterated Function Systems (IFSs) (cf.~\cite{Falc-Tech}), which include
     Sierpinski Triangle (ST),
    Sierpinski Carpet (SC), and hexagasket, to name a few fractals.
    Extending Galerkin method to nonlocal equations
    on fractal domains  and obtaining rate of convergence estimate similar to \eqref{rate-Lip}
    requires suitable adaptations of the Galerkin scheme and
    the generalized Lipschitz spaces to functions on self-similar domains. This is the main contribution
    of this work.

   % After a brief review of the convergence analysis of the Galerkin method for a nonlocal problem
   % on a unit $d$--cube (\S~\ref{sec.Euc}), we turn to the analysis of the fractal case. 
The outline of this paper is as follows.
In Section~\ref{sec.selfsim}, we collect basic facts about attractors of IFSs following \cite{Falc-Tech}
    and formulate a nonlocal diffusion equation on a self--similar domain 
\be\lbl{nfrac}
    \p_t u(t,x) = f(t,u) + \int_{K} W(x,y) D\left(u(t,y), u(t,x)\right) d\mu(y), \quad x\in K,
    \ee
    where compact $K\subset\R^d$ is an attractor of an IFS \cite{Hut81, Falc-Tech}
    equipped with self-similar probability measure $\mu$ \cite{Hut81, Falc-Tech, Kig01}).
    % The definitions of $K$ and $\mu$ will be given below.
    After that we develop a
    Galerkin scheme for the nonlocal equation on self--similar sets (\S~\ref{sec.Galerkin}).
     In Section~\ref{sec.frac},
    we estimate the rate of convergence of the $L^2$--projection of an $L^p$-function on a self--similar set, which
    is the key for estimating convergence of the Galerkin scheme.
    In Section~\ref{sec.integrate}, we review numerical
    methods for integrating functions on fractals, which is needed for the implementation of the Galerkin method.
    We conclude with a numerical example for a model problem on a ST in Section~\ref{sec.example}.

\section{The nonlocal equation on self-similar domains}\label{sec.selfsim}
\setcounter{equation}{0}

In this section, we extend the nonlocal equation \eqref{nloc} to fractal domains.
First, we review the IFSs to the extent needed to formulate
the assumptions on the spatial domain.  We follow \cite[Chapter~1]{Kig01} (see also \cite[\S~2.2]{Falc-Tech}).
After that we formulate the initial value problem and develop a Galerkin scheme
for the counterpart of \eqref{nloc} on a self--similar domain.

\subsection{Iterated function system}\label{sec.IFS}
Let $(X,\mathbf{d})$ be  a complete metric space and let $\{ F_1, F_2,\dots, F_d\}$  be a set of contractions on $X$
\begin{equation}\label{ifs}
\mathbf{d}\left(F_i(x), F_i(y)\right)\le \lambda_i \mathbf{d}(x,y), \quad x,y\in X,\; 0<\lambda_i<1, \; i\in [d].
\end{equation}
A unique compact $K\subset X$ satisfying
\begin{equation}\label{selfsim}
  K=\bigcup_{i=1}^d F_i(K)
\end{equation}
is called  an  attractor of IFS \eqref{ifs} (see \cite[Theorem~2.6]{Falc-Tech}).
The topological structure of $K$ is understood with the help of the symbolic space $\Sigma=S^\N$,
where $S=\{1,2,\dots, d\}$. On $\Sigma$, we define the shift map
$$
\Sigma \ni w=w_1w_2w_3\dots \; \mapsto \sigma(w)=w_2 w_3\dots\in\Sigma.
$$
The $i$th branch of the inverse of $\sigma$ is denoted by
$$
\sigma^{-1}_i(w_1w_2w_3\dots)=iw_1w_2w_3\dots ,\qquad i\in S.
$$
For a fixed $r\in (0,1)$, $\delta_r(w, v)=r^{\min\{m:\, w_m\neq v_m\}}$  defines a metric on $\Sigma$.
$(\Sigma, \delta_r)$ is a compact metric space (cf.~\cite[Theorem~1.2.2]{Kig01}).

Let $\pi:~\Sigma\rightarrow K$ be defined by
\begin{equation}\label{natural}
  \Sigma\ni w \mapsto \pi(w)=\bigcap_{k=1}^\infty F_{w_1w_2\dots w_k}(K)\,\in K,
\end{equation}
where  $F_w\doteq F_{w_1}\circ F_{w_2}\circ\dots\circ F_{w_m},\; w=w_1w_2\dots w_m.$
$\pi$ is a continuous surjective map \cite[Theorem~1.2.3]{Kig01}.

Bernoulli measure with weights
\begin{equation}\label{B-weights}
\mathbf{p}=(p_1,p_2,\dots, p_d),\qquad \sum_{i=1}^d p_i=1,\; p_i>0,
\end{equation}
is defined on cylinders $C_{c_1,c_2,\dots,c_l}=\{(w_1,w_2,\dots) \in\Sigma:\; w_{1}=c_1, w_{2}=c_2, \dots, w_{l}=c_l\} $,
$c_i\in S$, 
\begin{equation}\label{cylinders}
\nu(C_{c_1,c_2,\dots,c_p})=p_{c_1}p_{c_2}\cdot\dots\cdot p_{c_l},
\end{equation}
and is extended to a unique measure on $\left(\Sigma,\mathcal{B}\right)$. Here, $\mathcal{B}$ is a
Borel $\sigma$-algebra in $\Sigma$ generated by the collection of cylinders.

The Bernoulli measure can be characterized as a unique Borel regular probability measure on
$\Sigma$ that satisfies
$$
\nu(A) = \sum_{i\in S} p_i \nu\left( \sigma_i A\right)
$$
  for any Borel set $A\subset \Sigma$ \cite{Kig01}.

 A pushforward of $\nu$ yields a self--similar measure on $K$:
\begin{equation}\label{push}
  \mu(A)=\nu\left(\pi^{-1} A\right),\quad A\in\mathcal{N}=\{A\subset K:\; \pi^{-1}A\in \mathcal{M}\}.
\end{equation}
  We suppress the dependence of $\nu$ and $\mu$ on $\mathbf{p}$ to avoid overcomplicated notation.

  Theorem~1.4.5 in \cite{Kig01} gives necessary and sufficient conditions for
  \begin{equation}\label{mu-cylinder}
    \mu(K_w)=p_{w_1} p_{w_2}\dots p_{w_m},\qquad K_w=F_w(K), \; w=w_1w_2\dots w_m,
  \end{equation}
 where as before $F_w$ stands for $F_{w_1}\circ F_{w_2}\circ\dots\circ F_{w_m}$.
 These conditions are usually straightforward to check in practice.
  For instance, \eqref{mu-cylinder} holds if
  $\pi^{-1}(x)$ is a finite set for any $x\in K$ (cf.~\cite[Corollary~1.4.8]{Kig01}).
  Instead of stating the conditions from \cite[Theorem~1.4.5]{Kig01}, we postulate
  \eqref{mu-cylinder} as an assumption on $K$ as an attractor of an IFS. Alternatively, one
  can impose the Open Set Condition (see \cite{Falc-Tech}).
  
An important implication of \eqref{mu-cylinder} is the following property
  \begin{equation}\label{overlap}
    \mu\left(K_{wi}\bigcap K_{wj}\right)=0, \quad i\neq j, |w|=m, 
    \end{equation}
    for any $m=0,1,2,\dots$.
   Measure $\nu$ is ergodic with respect to the measure preserving shift map $\sigma$.
 We conclude our review of IFSs with a few examples of self--similar sets.  
   
    \begin{example}\label{ex.fractals}
      The following are three representative examples of self--similar sets.
      \begin{description}
      \item[(ST)]
        We begin with  the definition of ST.
        Let $X=\R^2$ and let   $v_1, v_2, v_3$ be vertices of an equilateral triangle
(see Fig.~\ref{f.2}\textbf{a}) and
$$
F_i=2^{-1}(x-v_i)+v_i,\; i\in [3].
$$
Then \eqref{selfsim} defines ST. The self-similar probability measure in this case satisfies
$$
  \mu(B) = \sum_{i=1}^2 \frac{1}{3} \mu\left(F^{-1}_i(B)\right),
$$
for every Borel $B\subset K$. Here, we took $p_1=p_2=p_3=1/3$. With this choice of $\mathbf{p}$, $\mu$ is
called a standard self--similar measure on $ST$ (cf.~\cite{Kig01}).
\item[(SC)]
  Let $X=\R^2$ and the set of $v_i,\;i\in[8]$ consists of $(0,0), (0,0.5), (0,1), (1, 0.5),
  (1,1), (0.5,1), (0,1), (0,0.5),$ $F_i=\frac{1}{3} (x+2v_i) $. This yields Sierpinski carpet
  (SC).
\item[(UC)]
Let $X=\R^d,$\, $v_i\in\{0,1\}^d, \; i\in [2^d],$ are the vertices of the unit square, and
$$
F_i=2^{-1}(x-v_i)+v_i,\; i\in [2^d].
$$
Then \eqref{selfsim} defines the $d$-cube. The self-similar measure in this case is the
Lebesgue measure. Thus, the IFS model naturally combines both Euclidean and fractal
sets as possible domains for the nonlocal diffusion equation.
  \end{description}
\end{example}

\subsection{The nonlocal equation}\label{sec.nloc}

Let $K$ be an attractor of an IFS equipped with a self--similar  probability
measure $\mu$ and consider 
\begin{align}\label{heat}
  \p_tu(t,x)& =f(t,u)+\int_K W(x,y) D\left(u(t,x),u(t,y)\right)d\mu(y), \\
  \label{heat-ic}
  u(0,x)& =g(x),\qquad x\in K,
\end{align}
where $W\in L^2(K\times K, \mu\times\mu)$, $g\in L^2(K,\mu)$ and $f, D$ are
subject to the following conditions.
We assume
 \be\lbl{bound-W}
 \max\left\{ \esssup_{x\in K} \int_K |W(x,y)| d\mu(y), \;
  \esssup_{y\in K} \int_K |W(x,y)| d\mu(x)\right\}<\infty.
\ee
Functions $f(t,u)$ and $D(u,v)$ are jointly continuous and 
\begin{align}\lbl{Lip-f}
  |f(t,u)-f(t,u^\prime)|\le L_f |u-u^\prime|,& \quad \forall t\in\R, \; u, u^\prime\in K,\\
  \lbl{Lip-D}
  |D(u,v)-D(u^\prime,v^\prime)|\le L_D \left(|u-u^\prime| +|v-v^\prime|\right),& \quad \forall
                                                                                 u,v, u^\prime, v^\prime\in K,
\end{align}
where $L_f$ and $L_D$ are positive Lipschitz constants.
In addition,
\be\lbl{bound-D}
\sup_{K\times K} |D(u,v)|\le 1.
\ee

\begin{theorem}\label{thm.well}
  Suppose \eqref{bound-W}-\eqref{bound-D} hold and $g\in L^2(K,\mu)$. Then for any $T>0$,
  the IVP \eqref{heat}-\eqref{heat-ic} has a unique solution $u\in C^1\left(0, T; L^2(K,\mu)\right)$.
\end{theorem}
\begin{proof} The proof is the same as in \cite[Theorem~3.1]{KVMed17} modulo minor notational
  adjustments.
  \end{proof}

  \subsection{Galerkin scheme for a nonlocal equation on a self-similar domain}
  \label{sec.Galerkin}

Next, we construct a Galerkin approximation of \eqref{heat}, \eqref{heat-ic}. First, we 
discretize $K$. To this end, we build self--similar partitions that are naturally associated with
$K$.
  Let $w=(w_1, w_2,\dots, w_m)\in\Sigma_m,$ $w_i\in S$, be an itinerary of length $m=|w|$ and
  consider the following partition of $K$:
  \begin{equation}\label{partition}
    \mathcal{K}^m\doteq \left\{K_w:\;K_w=F_w(K),\; |w|=m\right\},\qquad
  K=\bigcup_{|w|=m} K_w.
\end{equation}
% where $F_w\doteq F_{w_1}\circ F_{2} \circ\dots\circ F_{w_m}$.
Using $\mathcal{K}^m$, we define a finite-dimensional subspace of $\cH\doteq L^2(K,\mu\times\mu)$,
$\cH^m=\operatorname{span} \{ \1_{A}:\; A\in \mathcal{K}^m\}.$ Here and below, $\1_A$ is the
indicator of $A$. 
 In general, $K_w\bigcap K_v$ may be nonempty for two distinct $w,v\in \Sigma_m$,
 %The sets in $\{K_w:\; w\in\Sigma_m\}$ can be easily  modified to make them disjoint.
 %However, this is not necessary, because
 but $\mu (K_w\bigcap K_v)=0$
thanks to \eqref{overlap}. The  overlap  in support of the basis functions of $\cH^m$ does not interfere with the
convergence of the $L^2$-projections of the elements of $\cH$ onto $\cH^m$.
% Thus, we  will continue to work with unmodified sets $K_w$. 

Next, we approximate the solution of the IVP \eqref{heat}, \eqref{heat-ic} by
a piecewise constant function:
\begin{equation}\label{sum}
  u^m(t,x) = \sum_{|w|=m} u_w(t) \1_{K_w}(x).
\end{equation}
After inserting \eqref{sum} into \eqref{heat}
and projecting onto $\cH^m$ we arrive at the following
discretization of \eqref{heat}:
\begin{equation}\label{galerkin}
  \dot u_w = f(u_w,t)+ \sum_{|v|=m} W_{wv} D(u_w, u_v) \mu(K_v), 
  \end{equation}
  where
  \be\lbl{project-W}
  W_{wv}=\fint_{K_w\times K_v} W(x) d(\mu\times\mu)(x), \quad |w|=|v|=m.
  \ee
  System of ODEs \eqref{galerkin} can be rewritten as
  \be\lbl{m-heat}
  \p_t u^m(t,x)= f(u^m, t) +\int_K W^m (x,y) D\left(u^m(t,x), u^m(t,y)\right) d\mu(y),
  \ee
  where
  \be\lbl{def-Wm}
  W^m=\sum_{|w|,|v|=m} W_{wv} \1_{K_{wv}}, \quad  K_{wv}\doteq K_w\times K_v.
  \ee
  Similarly we approximate the initial condition \eqref{heat-ic} by
  \begin{equation}\label{approx-g}
    g^m= \sum_{|w|=m} g_w \1_{K_w},  \quad g_w = \fint_{K_w} g d\mu.
    \end{equation}
    \begin{theorem} Consider the IVPs for \eqref{heat}, \eqref{m-heat} subject to initial conditions
      \eqref{heat-ic} and \eqref{approx-g} respectively. Then for a given $T>0$, there exists $C>0$
      not dependent  on $m\in\N$ such that 
    \be\lbl{L2-conv}
    \sup_{t\in [0,T]} \|u^m(t,\cdot) -u(t,\cdot)\|_{L^2(K,\mu)} \le C\left( \|W^m-W\|_{L^2(K\times K,\mu\times\mu)}+
      \|u^m(0,\cdot)-u(0,\cdot)\|_{L^2(K,\mu)}\right).
    \ee
  \end{theorem}
  \begin{proof} The proof follows the lines of the proof the analogous statment in the Euclidean setting
    (cf.~\cite{KVMed18, Med19}).
    \end{proof}

\section{Approximation of functions on self--similar domains} \label{sec.frac}
    \setcounter{equation}{0}

    In view of \eqref{L2-conv}, our next step is to study convergence of the $L^2$--projections for functions defined
    on self--similar sets.
   As will be clear from the analysis of this section, the convergence estimates in this case
   depend on the geometry of the domain in general. The solution is more transparent in the case of
   self-affine
   sets, which already contains many interesting domains.
   Thus, in this section we restrict to
   % $X=\R^n$ and
    $F_i:X\to X, i\in [d],$ are affine contracting maps with the same contraction constant
    \begin{equation}\label{F-affine}
      \left|F_i(x)-F_i(y)\right| = \lambda \left|x-y\right|,\qquad 0< \lambda<1, \;\;x,y\in X.
    \end{equation}
    To this end, let $K$ be an attractor of the IFS \eqref{F-affine}. Assume that $K$ is equipped with
    a self-similar
    measure $\mu$, defined as a pushforward of the Bernoulli measure with weights
    $\mathbf{p}=(p_1,p_2,\dots, p_d)$ (see \eqref{cylinders}).
    % The sets in Example~\eqref{ex.fractals} satisfy the assumptions of this section.
    
For $i,j\in S, \; i\neq j, $ define  $\tau_{ij}$ as a unique vector such that 
$$
 T_{\tau_{ij}}\left( F_i \left(K\right)\right)
=F_j(K),
$$
where $T_{\tau_{ij}}(x)\doteq x+\tau_{ij}.$

Before we continue, we need to modify the definition of the generalized Lipschitz spaces to make
them applicable to the analysis of functions on attactors of IFSs.
\begin{definition}\label{df.modulus}
  For $\phi\in L^p(K,\mu), \; p\ge 1,$ define the $L^p$-modulus of continuity as
  follows
  $$
  \omega_p(\phi, \lambda, m)= \sup_{\ell\ge m}\max_{i\neq j}
  \left\{  \|\phi(\cdot+\tau)-\phi(\cdot)\|_{L^p(K^\prime_{\tau}, \mu)}: \quad
     \tau=\lambda^{-(\ell+1)}\tau_{ij},  i,j\in S \right\}, 
  $$
  where 
  $
  K^\prime_{\tau}=\{ x\in K: \; x+\tau \in K\}.
  $

  Further, for $\alpha\in (0,1],$  define the generalized Lipschitz space
$$
\operatorname{Lip}\left( L^p(K,\mu),\alpha,\lambda \right)=\left\{ \phi\in L^p(K,\mu):\; \exists 
C>0 :\; \omega_p(\phi,\lambda, m)\le C\lambda^{\alpha m} \right\}
$$
equipped with the norm
$$
\|\phi\|_{\Lip\left(L^p(K,\mu),\alpha, \lambda \right)}\doteq \sup_{m\in\N} \lambda^{-\alpha m}\omega_p(\phi,\lambda, m).
$$
\end{definition}

% \begin{remark} For each of the domains in the above table, one can reduce the set of $\mathcal{T}=\{\tau_{ij}\}$
% in the definition of the modulus of continuity to just two vectors
%   $tau_1$ and $\tau_2$ such that for any pair $(i,j)\in I^2$ there exists $c_i,c_j\in\Z$:
%   $$
%   \tau_{ij} = c_i\tau_1 +c_j\tau_2.
%   $$
% \end{remark}

The following lemma is the key result of this work.
\begin{lemma} Let $\phi\in \operatorname{Lip}\left(\alpha, L^p(K,\mu)\right), \; p\ge 1$.
  Then
  \be\lbl{Lp-rate}
  \|\phi^m-\phi\|_{L^p(K,\mu)} \le
  \frac{ d^{1/p} }{1-\lambda^\alpha}
  \|\phi\|_{\Lip\left(L^p(K,\mu),\alpha, \lambda \right)} \lambda^{\alpha m}.
  \ee
\end{lemma}
\begin{remark}
  It is instructive to compare the proof this lemma to the proof of
   its Euclidean counterpart (see Lemma~\ref{lem.cube} in the Appendix). 
\end{remark}
\begin{proof}
Fix $m\ge 1$ and partition $K$ into $d^m$ subsets
$$
K_{w}=F_{w}(K), \quad w \in \Sigma_m.
$$

As in the proof of Lemma~\ref{lem.cube}, we represent $\phi^{m+1}$ as 
  \be\lbl{phi-m+1}
  \phi^{m+1}=\sum_{|w|=m} \sum_{j\in S} \phi_{wj} \1_{K_{wj}},
  \ee
  where
  $$
  \phi_{wj} = \fint_{K_{wj}} \phi d\mu.
  $$
  Likewise,
  \begin{align}
    \nonumber
    \phi^{m}& =\sum_{|w|=m} \sum_{j\in S} \phi_{w} \1_{K_{wj}}\\
    \lbl{phi-m}
            &= \sum_{|w|=m} \sum_{j\in S}    \sum_{k\in S}p_k\phi_{wk} \1_{K_{wj}},
  \end{align}
  where we used $\phi_w= \sum_{k\in S}p_k\phi_{wk}$.

              By subtracting \eqref{phi-m+1} from \eqref{phi-m}, we have
              \begin{align}\nonumber
                  \phi^m(x)-\phi^{m+1}(x) & =p_k
\sum_{w,j,k}                                                           \left(
                                                         \fint_{K_{wj}}-\fint_{K_{wk}}\right) \phi(z) dz \,\1_{K_{wj}}(x)\\
                \label{represent}
                   & =  \sum_{w,j,k} p_k
               \fint_{K_{wj}} \left[ \phi(y+\tau^m_{jk})- \phi(y) \right] d\mu(y) \,\1_{K_{wj}}(x),
 \end{align}
where             
$$
\sum_{w,j,k}\doteq \sum_{|w|=m}\sum_{j\in S} \sum_{S\ni k\neq j}\quad\mbox{and}\quad
\tau^m_{ij}=\lambda^m\tau_{ij},\quad i,j\in S, \; i\neq j.
$$
After raising both sides of \eqref{represent} to the $p$th power and integrating over $K$, we have
\begin{equation}\label{raise-p}
\int_K |\phi^m-\phi^{m+1}|^p d\mu  =   \sum_{w,j,k} p_k^p
\left|\fint_{K_{wj}} \left[ \phi(y+\tau^m_{jk})- \phi(y) \right] d\mu(y) \right|^p  \mu(K_{wj}).
\end{equation}
Using the Jensen's inequality, from \eqref{raise-p} we obtain
\begin{align*}
\int_K |\phi^m-\phi^{m+1}|^p d\mu & \le  \sum_{w,j,k}
p_k^p \fint_{K_{wj}} \left|\phi(y+\tau^m_{jk})- \phi(y) \right|^p d\mu(y)  \mu(K_{wj})\\
& =  \sum_{j,k} p_k^p\left(\sum_{w}
                          \int_{K_{wj}} \left| \phi(y+\tau^m_{jk})- \phi(y) \right|^p d\mu(y)\right) \\
 &\le   C(d,\mathbf{p},p)\omega_p^p(\phi,\lambda, m),
\end{align*}
where
$$
C(d,\mathbf{p},p) =\sum_j\sum_{k\neq j} p_k^p\le \sum_j (1-p_j^p)\le d.
$$
From this we conclude
$$
\|\phi^m-\phi^{m+1}\|_{L^p(K,\mu)} \le d^{1/p} \omega_p(\phi,\lambda, m).
$$

 For any integer $M>m$ we have
\begin{equation}\lbl{bound-dyadic+}
\begin{split}
\|\phi^m-\phi^{m+M}\|_{L^p(K,\mu)} &
\le \sum_{k=m}^\infty  \left\|\phi^k-\phi^{k+1}\right\|_{L^p(K,\mu)}\\
&\le d^{1/p}\|\phi\|_{\Lip\left(L^p(K,\mu), \alpha,\lambda \right)}
\frac{\lambda^{\alpha m}}{1-\lambda^\alpha}.
\end{split}
\end{equation}
By passing $M$ to infinity in \eqref{bound-dyadic+}, we get \eqref{Lp-rate}.
\end{proof}

\begin{remark}\label{rem.product}
  For \eqref{L2-conv}, we also need error estimate for $W\in L^2(K\times K,\mu\times\mu)$.
  To this end, we note that
  under the assumptions made in this section, $K\times K$ is a self-similar subset of $X\times X$
  generated by $d^2$
  contractions $F_{ij}=(F_i, F_j), \; i,j\in S$ with the same contraction constant $\lambda$ as before. Therefore,
  \eqref{Lp-rate} still holds after replacing $d$ by $d^2$.
  \end{remark}

  \section{Integration of functions on self--similar domains}\label{sec.integrate}
  \setcounter{equation}{0}
  Numerical implementation of the Galerkin scheme \eqref{galerkin} requires computation of the
  integrals of the form
  \begin{equation}\label{integrals}
  \fint_{K_{wv}} W d(\mu\times\mu) \quad \mbox{and}\quad \fint_{K_w} gd\mu \qquad w,v\in\Sigma_m.
  \end{equation}
  Since $\mu$ is a self-similar measure defined as a pushforward of the Bernoulli measure,
a few comments on practical evaluation of \eqref{integrals} are in order.

  In this section, we review three algorithms of numerical evaluation of \eqref{integrals}. They are
  based on two  sets of ideas.  The first one is based on ergodicity of $\mu$ \cite{Falc-Tech},
  which can be used to construct Monte-Carlo type approximations. The second set is related to
  uniform partitions that are naturally associated with self--similar sets \cite{InfVol09}.
  This leads to quasi-Monte-Carlo type methods. For the latter, we estimate
  the rate of convergence of numerical approximation of \eqref{integrals} assuming some regularity,
  for instance, H\"{o}lder
  continuity of the integrand.

  By the reasons explained in Remark~\ref{rem.product}, it is sufficient to discuss the problem
  of evaluation of \eqref{integrals} only for
  $\fint_{K_w}\phi d\mu$ with $\phi\in L^1(K,\mu)$. Furthermore, since
  $$ \frac{\mu (\cdot)}{\mu (K_w)} $$
  is a probability measure on $K_w$, without loss of generality we may consider
  \begin{equation}\label{evaluate}
  \int_K \phi d\mu, \qquad \phi\in L^1(K, \mu).
  \end{equation}
  Therefore, in the remainder of this section we will focus on the problem of evaluation of
  \eqref{evaluate}.
  To this end, let $K$ be an attractor of a system of contractions (cf.~\eqref{ifs}).
  % As an
  % ambient space we take $\R^n$ as in Section~\ref{sec.frac}. Any other complete metric space would work
  % as well.
  We assume that $F_i$'s satisfy the Open Set Condition (cf.~\cite{Hut81, Kig01}), which, in particular,
  implies \eqref{overlap} (cf.~\cite{Hut81}).

  The first algorithm relies on the following lemma.
  \begin{lemma}\label{lem.erg} Let $\phi\in L^1(K, \mu)$. Then
    \begin{equation}\label{erg}
\lim_{m\to\infty} \frac{1}{m} \sum_{j=0}^{m-1} \phi\left(\pi (\sigma^j s)\right) =\int_K \phi (x) d\mu(x).
      \end{equation}
    \end{lemma}
    \begin{proof} Since $\nu$ is an ergodic probability measure, by the Ergodic Theorem
      (cf.~\cite[Theorem~6.1]{Falc-Tech}), we have
      \begin{align*}
        \lim_{m\to\infty} \frac{1}{m} \sum_{j=0}^{m-1} \phi\left(\pi (\sigma^j s) \right) &=\int_\Sigma \phi\circ\pi (y) d\nu (y)\\        
                                                                                         &=\int_K \phi (x) d\mu(x).
      \end{align*}
      \end{proof}

      Lemma~\ref{lem.erg} suggests the following algorithm for evaluating \eqref{evaluate}
      in the spirit of the Monte-Carlo method . 
 
    \noindent\textbf{Algorithm I.} 
    Let    $\phi\in L^1(K,\mu)$.
    \begin{description}
    \item[-] Generate a random string of length $2N$: $s=(s_1,s_2,\dots,s_{2N})\sim\operatorname{Uniform}~(\Sigma_{2N})$.
\item[-] Set
  $$
 S_N:= \frac{1}{N} \sum_{i=1}^N \phi\left(\pi(\sigma^{i-1}s)\right).
 $$
 \item[-] Use $S_N$ to approximate $\int_K \phi d\mu$.
\end{description}

Algorithm~I is essentially a Monte-Carlo method.
% convergence rate $O(N^{-1/2})$.
As usual with the Monte-Carlo schemes, the convergence is slow but it depends neither on the dimension of
the spatial domain nor does it require any regularity
of $\phi$ beyond integrability. For continuous functions, one can use a
quasi-Monte-Carlo type algorithm based on  uniform
sequences (cf.~\cite{InfVol09}).

\noindent\textbf{Algorithm II.} 
\begin{description}
\item[-] Pick $x_0\in K$.
\item[-] Compute $x_w=F_w(x_0), \quad w\in\Sigma_m$.
\item[-] Set $N=d^m$, $S_N:= \frac{1}{N} \sum_{|w|=m} \phi( x_w).$
\end{description}

\begin{lemma}\label{lem.quasi} For $f\in C(K)$,
  \begin{equation}\label{quasi}
\lim_{m\to\infty} \left|S_N-\int_K f\, d\mu\right|=0.
    \end{equation}
  \end{lemma}
\begin{proof}
  Points $\{x_w,\; |w|=m\}$ form a uniform sequence (cf.~\cite[Theorem~2.5]{InfVol09}).
\end{proof}

Furthermore, Theorem~3.1 in \cite{InfVol09} gives an estimate of the discrepancy, which
together with Koksma--Hlawka inequality would imply $O(N^{-1})$ convergence in \eqref{quasi}.
We are not aware of an analog of the Koksma-Hlawka inequality applicable to the problem
at hand\footnote{There is a variant of Koksma-Hlawka inequality for functions on fractals
  in \cite{MalStr18}, but it relies on harmonic splines, and does not apply to our setting.}.
% The result in \cite{MalStr18} requires computing harmonic splines \cite{StrUsh2000}.
% It is not clear how to apply
% it in our setting.
%and the discrepancy estimate in Theorem~3.1 \cite{InfVol09} yield
%\begin{equation}\label{quasi}
%  \left|S_N-\int_K\phi d\mu\right| =O(N^{-1})
%  \end{equation}
%  for any $\phi$ of bounded variation.
On the other hand, if one allows for a little  more regularity of $\phi$, the rate of convergence
for a slight modification
of Algorithm~II can be estimated. For simplicity, we formulate the algoritm and the convergence
result for functions ST used in the numerical example in the following section.
% The generalization of these ideas for functions on attractors of IFSs is straightforward.

\noindent\textbf{Algorithm III.} Let $T$ stand for an equilateral triangle with sides of length $1$ and
denote $T_w=F_w(T),\; w\in\Sigma_m$. Denote the vertices of $T_w$ by
$$
q^w_i=\pi(w\bar{i}),\quad  \phi_i^w = \phi( q_i), i\in [3].
$$
Set
\begin{align}\label{approx-int}
  S_N \doteq\frac{1}{3^{m+1}}  \sum_{|w|=m} \sum_{j=1}^3  \phi^w_j, \qquad N=3^m.
  \end{align}

  \begin{lemma}\label{lem.approx}
    Suppose $K$ is a ST and $\phi \in C(K)$ is a twice continuously differentiable function in directions
    $\tau_i, \; i\in [3]$.
    Then
    \begin{equation}\label{geom-rate}
      \left|  \int_K \phi d\mu - S_N\right|=O(h^2),\qquad h=2^{-m}.
    \end{equation}
    If $\phi$ is only H\"{o}lder continuous with exponent $\alpha\in (0,1]$. Then
    \begin{equation}\label{Holder-rate}
      \left|  \int_K \phi d\mu - S_N\right|=O(h^\alpha).
    \end{equation}
  \end{lemma}
  \begin{remark}\label{rem.lambda} It is evident from the proof of the lemma that
    \eqref{Holder-rate} extends to H\"{o}lder continuous functions on attaractors of IFS \eqref{F-affine}:
  \begin{equation}\label{Holder-general}
      \left|  \int_K \phi d\mu - S_N\right|=O(\lambda^{\alpha m}), \; N=d^m.
    \end{equation} 
  \end{remark}
  
 \begin{proof} 
Denote
$$
q^w_{ij}= \frac{1}{2} \left( q^w_i+q^w_j\right)\quad \mbox{and}\quad   \phi^w_{ij}=\phi(q^w_{ij})\quad 0<i<j\le 3.
$$
By Taylor's formula, for $h=2^{-(m+1)}$ and $0<i<j\le 3$, we have
\begin{align*}
  \phi (q^w_i)& = \phi (q^w_{ij}) - D_{\tau_{ij}}\phi (q^w_{ij}) h +\frac{1}{2} D^2_{\tau_{ij}}\phi (\xi^w_{ij}) h^2,\\
  \phi (q^w_j)& = \phi (q^w_{ij}) + D_{\tau_{ij}}\phi (q^w_{ij}) h +\frac{1}{2} D^2_{\tau_{ij}}\phi (\eta^w_{ij}) h^2,
\end{align*}
where $\xi^w_{ij}, \eta^w_{ij}$ lie in the segment connecting $q^w_i$ and $q^w_j$, and
$D_{\tau}$ stands for the directional derivative. From here we conclude
that
\begin{equation}\label{Taylor}
  2 \phi (q^w_{ij}) =  \phi (q^w_i) +  \phi (q^w_j) + O(h^2).
\end{equation}
Recall
$$
S_m=\frac{1}{3^{m+1}} \sum_{|w|=m}\sum_{j=1}^3 \phi^w_j
$$
and note that
\begin{align} \nonumber
  S_{m+1}& =\frac{1}{3^{m+2}} \sum_{|w|=m} \left( \sum_{j=1}^3 \phi^w_j +\sum_{k\neq l} \phi^w_{kl}\right)\\
  & = \frac{1}{3^{m+1}} \sum_{|w|=m} \left( \sum_{j=1}^3 \phi^w_j  + O(h^2)\right)
  \label{martingale}
       = S_m+O(h^2),
\end{align}
where we used \eqref{Taylor} to obtain the second equality in \eqref{martingale}.
From \eqref{martingale}, recalling $h=2^{-(m+1)}$, we conclude that $\{S_m\}$ is a convergent
sequence with the rate specified in \eqref{geom-rate}.

To prove \eqref{Holder-rate}, replace \eqref{Taylor} with
$$
  2 \phi (q^w_{ij}) =  \phi (q^w_i) +  \phi (q^w_j) + O(h^\alpha)
$$
and continue as above.

\end{proof}

\section{Numerical example}\label{sec.example}
\setcounter{equation}{0}

In this section, we present numerical results illustrating analysis in the previous sections.
To this end, consider the following IVP
\begin{align}
  \label{nKM}
  \partial_t u(t,x) &= \int_K W(x,y) \left( u(t,y)-u(t,x)\right) d\mu(y),\\
                      \label{nKM-ic}
                      u(0,x) &=\left\{
                               \begin{array}{ll} 1, & x\in F_2(K),\\
                                 -1, & x\in F_1(K)\bigcup F_3(K),
                               \end{array}
                                      \right.
\end{align}
where $W(x,y)=\exp\{-2|x-y|^2\}$ and $K$ is ST (cf.~Example~\ref{ex.fractals}). We use the same notation for $W$
as a function on $\R^2$ and its restriction to $K$.

Next we discretize \eqref{nKM}, \eqref{nKM-ic} as explained in \S~\ref{sec.Galerkin}.
Specifically, we approximate $u(t,x)$ by a piecewise constant function in $x$, $u^m(t,x)$
(cf.~\eqref{sum}), and form a system of ODEs \eqref{galerkin}. The latter is solved by
the fourth order Runge-Kutta method with temporal step $\Delta t=10^{-3}$.
The spatial profile of the numerical solution in Figure~\ref{f.3}a.
To better visualize solution, we plot $u^m\left(t, h(x)\right)$ with 
$$
h: K\ni x\mapsto \bar x \in [0,1]
$$
such that
\begin{align*}
  x&=F_w (K), \quad w\in\Sigma,\\
  \bar x&= \sum_{i=1}^\infty \frac{w_i-1}{d^i}, \qquad\qquad \mbox{(cf.~\S\ref{sec.IFS})}.
\end{align*}
Note that $h$ is a measure preserving map from $(K,\mu)$ to $[0,1]$ equipped with
Lebesgue measure.
\begin{figure}
	\centering
 \textbf{a}\;	\includegraphics[width =.45\textwidth]{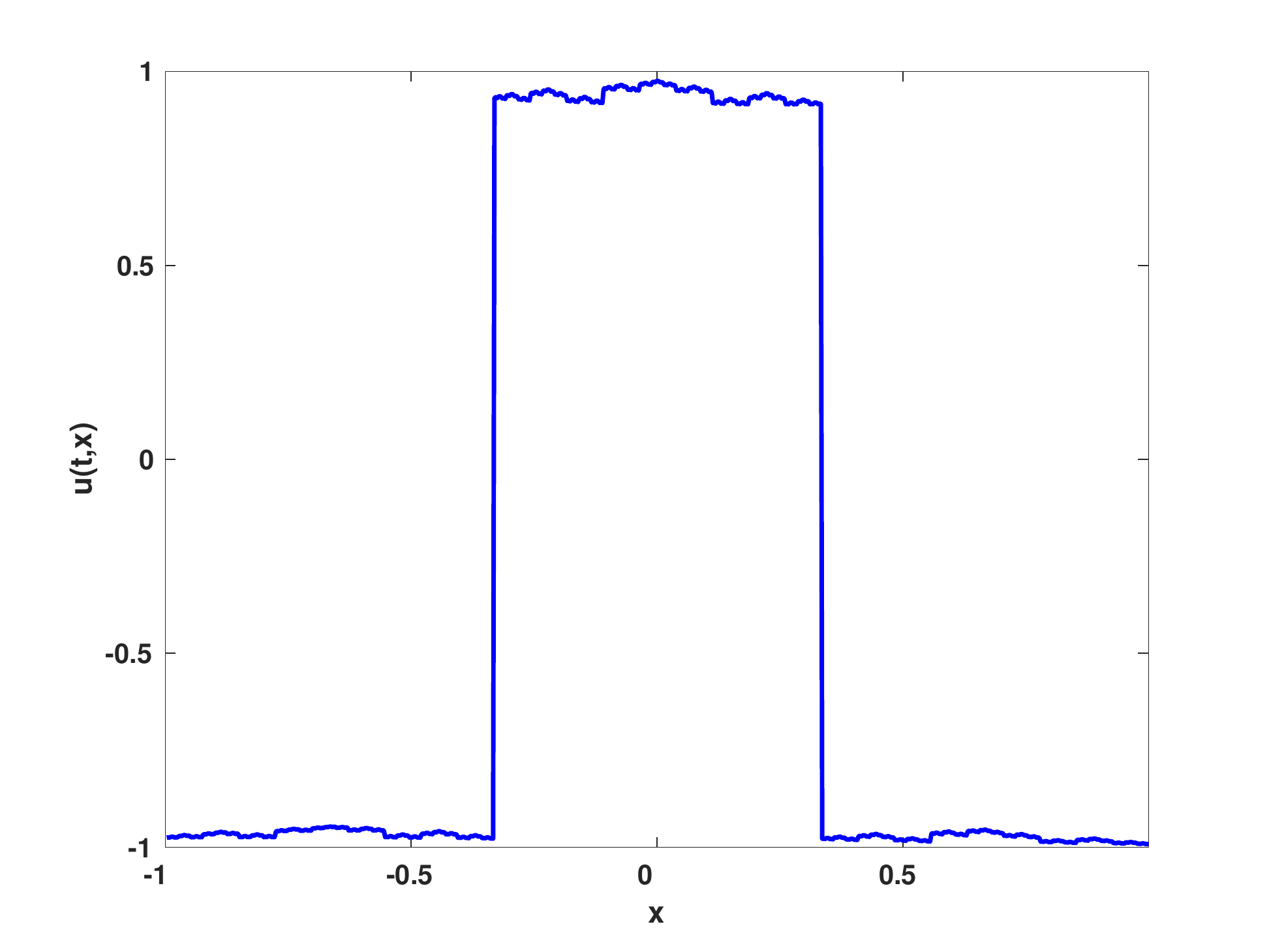}\qquad
  \textbf{b}\;      \includegraphics[width = .45\textwidth]{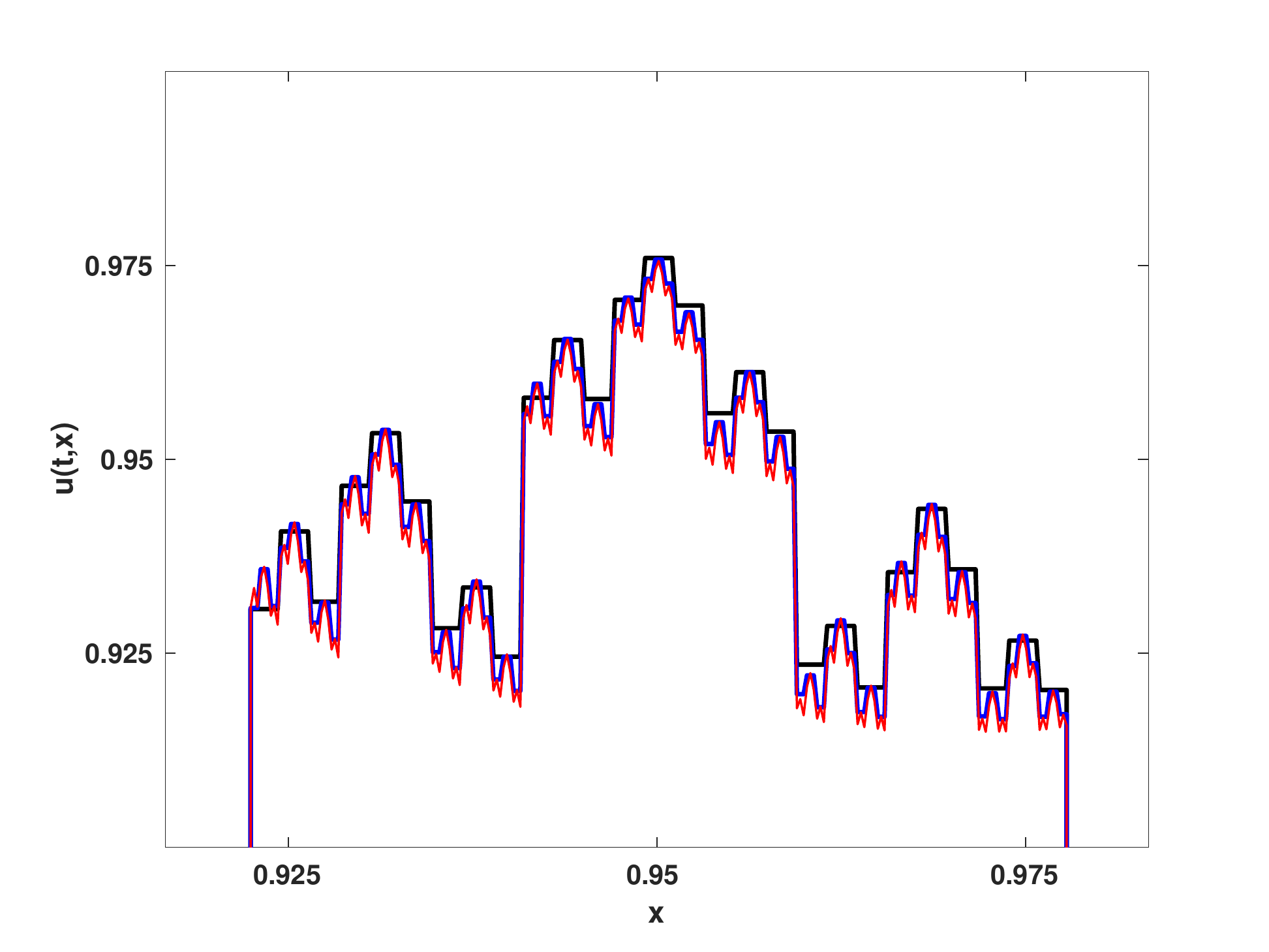}
  \caption{  \textbf{a}) Numerical solution of the IVP \eqref{nKM}, \eqref{nKM-ic} computed for $t=0.1$ and $m=6$.
    \textbf{b}) Numerical approximations $u^m(0.1,x)$ computed with $m=4$ (black), $m=5$ (blue), and $m=6$ (red).
  }
	\label{f.3}
\end{figure}

We chose a higher order Runge--Kutta scheme and integrated until a small time $t=0.1$ to discount
for the error due to the temporal discretization. The latter can be analyzed by standard techniques
for numerical integration of ODEs (see \cite[Section~6]{KVMed22} for the analysis of a closely related problem).
For the model at hand, the error of temporal discretization is dominated by the error of the Galerkin
scheme analyzed above. To estimate the rate of convergence of the numerical scheme, we assume
$$
\Delta^l\doteq \| u^{l+1}(t,\cdot)-u^{l}(t,\cdot)\|_{L^2(K,\mu)}\sim C\lambda^{\alpha l},
$$
as implied by \eqref{L2-conv}, \eqref{Lp-rate}. Then after computing $\Delta^l$ and
$\Delta^{l+1}$  we estimate
\begin{equation}\label{est-alpha}
  \alpha^l=\frac{\log\Delta^{l+1} - \log\Delta^{l}}{\log\lambda}.
\end{equation}  

Numerical estimates
$$
\alpha^3\approx 1.002, \alpha^4\approx 1.015, \alpha^5\approx 1.007 
$$
agree well with the theoretical estimate $\alpha=1$ obtained for smooth $W$ (cf.~\eqref{Lp-rate}).
It would be interesting 
to test convergence rate for $W\in\operatorname{Lip}\left(\alpha, L^p(K,\mu)\right)$ with fractional $\alpha\in (0,1)$.
% However, at the moment, we do not know a good example of such function.

\appendix
\section{Approximation of functions on the unit cube}\label{sec.Euc}
    \setcounter{equation}{0}

    In this appendix, we include the proof of  \eqref{rate-Lip} for functions defined on the unit $d-$cube $Q=[0,1]^d$, i.e.,
    in the Euclidean setting (cf.~\cite[Section~5]{KVMed22}). The generalization of this proof yields the rate of
    convergence estimate for functions
    on self--similar domains (see Section~\ref{sec.frac}).

    Let
    $f\in\Lip(L^p(Q), \alpha),\; \alpha\in (0,1], \; 1\le p<\infty,$ and approximate it by
    \be\lbl{fn}
    f^n(x) =\sum_{i\in [n]^d} \fint_{Q_i} f(z)dz \1_{Q_i}(x),
    \ee
    where
    $$
     Q_i=
    \left[\frac{i_1-1}{n},\frac{i_1}{n}\right)\times \dots\times\left[\frac{i_d-1}{n},\frac{i_d}{n}\right),\; i=(i_1,i_2,\dots,i_d).
    $$

    We want to estimate $ \|f-f^n\|^p_{L^p(Q)},\; p\ge 1.$ To utilize  self-similar structure of the unit cube,
    below we use a dyadic discretization of $Q$, i.e.,  $n:=2^m, \; m\in\N.$
    To avoid excessively cumbersome notation, for the remainder of this section, we denote $f^m:=f^{2^m}$.
    
 \begin{lemma}\lbl{lem.cube}
    Let  $f\in\Lip(L^p(Q), \alpha),\; \alpha\in (0,1], \; 1\le p<\infty.$ Then 
      \begin{equation}\label{rate-cube}  
        \|f-f^m\|^p_{L^p(Q)} \le C  \|f\|_{\Lip(L^p(Q),\alpha)} 2^{-\alpha m},
        \quad C=\frac{(2^d-1)^{1/p} 2^\alpha d^{\alpha/2} }{2^\alpha-1}.
       \end{equation}
       \end{lemma}
       \begin{proof}  Let $h=2^{-m}$ and denote
         \begin{align*}
          Q^\prime_i & =
          \left[\frac{i_1-1}{n},\frac{i_1-1}{n} +\frac{h}{2}\right)\times \dots\times\left[\frac{i_d-1}{n},\frac{i_d-1}{n}+\frac{h}{2}\right),\\
           Q_{ij} &= Q_i^\prime+\frac{h}{2} j, \quad  i\in [n]^d,\; j\in\{0,1\}^d.
         \end{align*}
         The right-hand side of the definition of $Q_{ij}$ is interpreted as the translation of $Q_i^\prime$ by $\frac{h}{2} j$.

         With these definitions, we can express
         \begin{equation}\label{fm+1}
         f^{m+1}(x) =\sum_{i\in [n]^d}\sum_{j\in \{0,1\}^d} \fint_{Q_{ij}} f(z)dz \1_{Q_{ij}}(x).
         \end{equation}
         In the remainder of the proof, the ranges of the indices $i$ and $j$ will remain $[n]^d$ and $\{0,1\}^d$ respectively.
         They  will be omitted for brevity.
          \begin{align} \nonumber
            f^m (x) &=\sum_i \fint_{Q_i} f(z)dz \1_{Q_i}(x)\\
            \nonumber
                    & = \sum_i \left\{\left(2^{-d} \sum_{k\in\{0,1\}^d} \fint_{Q_{ik}} f(z) dz \right)
                      \left(\sum_j  \1_{Q_{ij}}(x)\right)\right\}\\
            \label{fm}
                  &= 2^{-d} \sum_{i,j,k} \fint_{Q_{ik}} f(z) dz \1_{Q_{ij}}(x).
          \end{align}
          From \eqref{fm+1} and \eqref{fm}, we have
          $$
          2^d  \left(f^{m+1}(x)-f^m(x)\right)= \sum_{i,j}\sum_{k\neq j} \left(\fint_{Q_{ij}}-\fint_{Q_{ik}}\right)f(z) dz \1_{Q_{ij}}(x).
          $$
          We continue with the following estimates
          \begin{align} \nonumber
            2^{dp} \| f^{m+1}-f^m\|_{L^p(Q)}^p & \le  2^{-d(m+1)}\sum_{i,j, k\neq j} \left|
                                                \fint_{Q^\prime_i} \left(f (s+\frac{h}{2}j) - f (s+\frac{h}{2}k) \right) ds \right|^p\\
            \nonumber
                                               &\le  2^{-d(m+1)}\sum_{i, k\neq j}
                                                 \left(\sum_i \fint_{Q^\prime_i} \left|f (s+\frac{h}{2}j) - f (s+\frac{h}{2}k) \right|^p ds\right)\\
            \nonumber
                                               & \le \sum_{i, k\neq j} \int_{Q^\prime}  \left|f (s) - f (s+\frac{h}{2}(k-j) ) \right|^p ds\\
            \label{m-step}
                                               & \le 2^d(2^d-1)\omega_p^p \left(f, \sqrt{d} \frac{h}{2}\right).
          \end{align}
          
          From \eqref{m-step}, we have
          \begin{equation}\label{mstep+}
          \| f^{m+1}-f^m\|_{L^p(Q)} \le (2^d-1)^{1/p} \|f\|_{\Lip(L^p(Q),\alpha)} d^{\alpha/2} 2^{-\alpha(m+1)}.
        \end{equation}
        We conclude that $\{f^m\}$ is a Cauchy sequence and
        $$
        \| f^m-f\|_{L^p(Q)}  \le \sum_{k=m}  \| f^{k+1}-f^k\|_{L^p(Q)} \le  \frac{(2^d-1)^{1/p} 2^\alpha d^{\alpha/2} }{2^\alpha-1}
        \|f\|_{\Lip(L^p(Q),\alpha)} 2^{-\alpha m}.
        $$
         
\end{proof}

\vskip 0.2cm
\noindent
{\bf Acknowledgements.} 
The author is grateful to Alexander Teplyaev for very helpful discussions
on the Galerkin scheme in the fractal setting.
This work was partially supported by NSF grant DMS 2009233.

\bibliographystyle{amsplain}
% \bibliography{nonloc}

\begin{thebibliography}{10}

\bibitem{BCCZ19}
Christian Borgs, Jennifer~T. Chayes, Henry Cohn, and Yufei Zhao, \emph{An
  {$L^p$} theory of sparse graph convergence {I}: {L}imits, sparse random graph
  models, and power law distributions}, Trans. Amer. Math. Soc. \textbf{372}
(2019), no.~5, 3019--3062.

\bibitem{Cha17}
S.~Chatterjee, \emph{Large deviations for random graphs}, Lecture Notes in
  Mathematics, vol. 2197, Springer, Cham, 2017, Lecture notes from the 45th
  Probability Summer School held in Saint-Flour, June 2015, \'{E}cole
  d'\'{E}t\'{e} de Probabilit\'{e}s de Saint-Flour. [Saint-Flour Probability
  Summer School]. 

\bibitem{Falc-Tech}
Kenneth Falconer, \emph{Techniques in fractal geometry}, John Wiley \& Sons,
  Ltd., Chichester, 1997. \MR{1449135}

\bibitem{Falc-FracGeom}
\bysame, \emph{Fractal geometry}, third ed., John Wiley \& Sons, Ltd.,
  Chichester, 2014, Mathematical foundations and applications. 

\bibitem{Hut81}
John~E. Hutchinson, \emph{Fractals and self-similarity}, Indiana Univ. Math. J.
  \textbf{30} (1981), no.~5, 713--747. 

\bibitem{InfVol09}
Maria Infusino and Aljo\v{s}a Vol\v{c}i\v{c}, \emph{Uniform distribution on
  fractals}, Unif. Distrib. Theory \textbf{4} (2009), no.~2, 47--58.
  

\bibitem{KVMed17}
Dmitry Kaliuzhnyi-Verbovetskyi and Georgi~S. Medvedev, \emph{The semilinear
  heat equation on sparse random graphs}, SIAM J. Math. Anal. \textbf{49}
  (2017), no.~2, 1333--1355. 

\bibitem{KVMed18}
\bysame, \emph{The {M}ean {F}ield {E}quation for the {K}uramoto {M}odel on
  {G}raph {S}equences with {N}on-{L}ipschitz {L}imit}, SIAM J. Math. Anal.
  \textbf{50} (2018), no.~3, 2441--2465. 

\bibitem{KVMed22}
\bysame, \emph{Sparse {M}onte {C}arlo method for nonlocal diffusion problems},
  SIAM J. Numer. Anal. \textbf{60} (2022), no.~6, 3001--3028. 

\bibitem{Kig01}
Jun Kigami, \emph{Analysis on fractals}, Cambridge Tracts in Mathematics, vol.
  143, Cambridge University Press, Cambridge, 2001. 
\bibitem{MalStr18}
Jens Malmquist and Robert~S. Strichartz, \emph{Numerical integration for
  fractal measures}, J. Fractal Geom. \textbf{5} (2018), no.~2, 165--226.

\bibitem{Med14a}
Georgi~S. Medvedev, \emph{The nonlinear heat equation on dense graphs and graph
  limits}, SIAM J. Math. Anal. \textbf{46} (2014), no.~4, 2743--2766.

\bibitem{Med14b}
\bysame, \emph{The nonlinear heat equation on {$W$}-random graphs}, Arch.
  Ration. Mech. Anal. \textbf{212} (2014), no.~3, 781--803. 

\bibitem{Med19}
\bysame, \emph{The continuum limit of the {K}uramoto model on sparse random
  graphs}, Communications in Mathematical Sciences \textbf{17} (2019), no.~4,
  883--898.

\bibitem{MedSim22}
Georgi~S. Medvedev and Gideon Simpson, \emph{A numerical method for a nonlocal
  diffusion equation with additive noise}, Stochastics and Partial Differential
  Equations: Analysis and Computations (2022).

\bibitem{Nikol-approximation}
S.~M. Nikolski\u{\i}, \emph{Approximation of functions of several
  variables and imbedding theorems}, Die Grundlehren der mathematischen
  Wissenschaften, Band 205, Springer-Verlag, New York-Heidelberg, 1975,
  Translated from the Russian by John M. Danskin, Jr. 

\bibitem{StrUsh2000}
Robert~S. Strichartz and Michael Usher, \emph{Splines on fractals}, Math. Proc.
  Cambridge Philos. Soc. \textbf{129} (2000), no.~2, 331--360. 

\bibitem{WilStr06}
D.A. Wiley, S.H. Strogatz, and M.~Girvan, \emph{The size of the sync basin},
  Chaos \textbf{16} (2006), no.~1, 015103, 8. 

\end{thebibliography}
\def\cprime{$'$} \def\cprime{$'$} \def\cprime{$'$}
\providecommand{\bysame}{\leavevmode\hbox to3em{\hrulefill}\thinspace}
\providecommand{\MR}{\relax\ifhmode\unskip\space\fi MR }
\providecommand{\MRhref}[2]{%
  \href{http://www.ams.org/mathscinet-getitem?mr=#1}{#2}
}
\providecommand{\href}[2]{#2}

\end{document}